\documentclass{amsart}
\usepackage{amsmath,amssymb,mathtools}
\usepackage[colorlinks=true,linkcolor=blue,citecolor=blue,urlcolor=blue]{hyperref}

\newtheorem{theorem}{Theorem}[section]
\newtheorem{lemma}[theorem]{Lemma}
\newtheorem{proposition}[theorem]{Proposition}

\theoremstyle{remark}
\newtheorem{remark}[theorem]{Remark}

\DeclareMathOperator{\SL}{SL}
\DeclareMathOperator{\Span}{span}
\DeclareMathOperator{\rank}{rank}
\newcommand{\Z}{\mathbb Z}
\newcommand{\R}{\mathbb R}
\newcommand{\C}{\mathbb C}
\newcommand{\HH}{\mathbb H}
\newcommand{\vt}{\vartheta}

\title[Reverse Minkowski for Gaussian Mass of Integral Unimodular Lattices]
{A Sharp Reverse Minkowski Inequality for the Gaussian Mass of Integral Unimodular Lattices Through Rank $32$}

\author[Scott Duke Kominers]{Scott Duke Kominers}

\address{Harvard Business School; Department of Economics and Center of
Mathematical Sciences and Applications, Harvard University; and a16z crypto}

\email{kominers@fas.harvard.edu}

\thanks{I used LLMs to assist with computations, analysis, and synthesis in
the preparation of this article, particularly GPT-5.5 Pro and Claude 4.7 Opus
(accessed in part via Poe with the support of Quora, where I am an advisor).
I particularly appreciate a thorough review from Refine.ink. The problem,
methods, and eventual written form are my own; and of course any errors remain
my responsibility. This work was conducted while I was visiting the
Technological Innovation, Entrepreneurship, and Strategic Management (TIES)
Group at the MIT Sloan School of Management; I greatly appreciate their
hospitality.}

\begin{document}

\begin{abstract}
The integer lattice $\Z^n$ is conjectured to maximize the Gaussian mass
$\Theta_L(t)=\sum_{x\in L}e^{-t\|x\|^2}$ over the set of stable lattices
in~$\R^n$, for every $t>0$.
We prove this sharp inequality for every integral unimodular lattice $L$ of
rank $n\leq 32$, with equality only at $L\cong\Z^n$, and furthermore obtain
the strict inequality for every even unimodular lattice of rank~$40$.
The proof does not use the classification of unimodular lattices in these
ranks; rather, it parametrizes integral unimodular theta series as
polynomials in the modular function $u=\Delta_8/\vt_3^8\in(0,1/64]$, with
the few coefficients that arise controlled by norm-$1$ splitting, ADE root
counts, and shadow positivity.
\end{abstract}

\maketitle

\section{Introduction}\label{sec:intro}

The \emph{Gaussian mass} of a full-rank lattice $\Lambda\subset \R^n$ is
\[
  \Theta_\Lambda(t)=\sum_{x\in\Lambda} e^{-t\|x\|^2},\qquad t>0.
\]
A \emph{reverse Minkowski inequality} goes in the opposite direction to
Minkowski's classical bound: where Minkowski's theorem guarantees that a
lattice of small covolume contains short nonzero vectors, a reverse
Minkowski inequality upper-bounds a measure of how concentrated a lattice can
be near the origin---here, the Gaussian mass $\Theta_\Lambda(t)$---under an
appropriate normalization on the lattice.

A unit-covolume lattice is called \emph{stable} if every nonzero sublattice
has covolume at least $1$ in its linear span; this notion was introduced by
Stuhler \cite{Stuhler1976} and refined by Grayson \cite{Grayson1984} as a
lattice-theoretic analogue of semistability for vector bundles (for a modern
exposition, see \cite{Casselman}).
Dadush conjectured a nearly-sharp reverse Minkowski inequality for stable
lattices, which Regev and Stephens-Davidowitz \cite{RSD} proved (see also
\cite{DadushRegev}), building on the lattice Gaussian-mass technology of
Banaszczyk~\cite{Banaszczyk}.
In the course of their proof, Regev and Stephens-Davidowitz also established
the sharp pointwise inequality
$\Theta_\Lambda(t)\leq\Theta_{\Z^n}(t)$ for stable $\Lambda$ at sufficiently
extreme parameters~$t$.
The sharp form
\begin{equation}\label{eq:RSD-conjecture}
  \Theta_\Lambda(t)\leq \Theta_{\Z^n}(t)
\end{equation}
for every stable $\Lambda\subset\R^n$ and every $t>0$---equivalently, the
assertion that the integer lattice maximizes Gaussian mass over stable
lattices at every scale---remains open in general.
Eisenberg, Regev, and Stephens-Davidowitz \cite{ERSD} established the
corresponding sharp inequality for a family of Epstein-zeta functions,
building on Sarnak and Str\"ombergsson's work \cite{SarnakStrombergsson} on
local extrema of the Epstein zeta function and heights of flat tori.

In this paper, we show the inequality \eqref{eq:RSD-conjecture} for every
integral unimodular lattice $L$ of rank $n\leq 32$, with equality only at
$L\cong\Z^n$, as well as for every even unimodular lattice $L$ of rank $40$.
(Recall that a lattice $L$ is \emph{even} if $\|x\|^2\in 2\Z$ for every
$x\in L$, and \emph{odd} otherwise.)
Every positive-definite integral unimodular lattice $L$ is stable.
Indeed, such $L$ have covolume $1$ since they are unimodular, and if
$M\subset L$ is a nonzero sublattice, then the square of the covolume of $M$
in $\Span_{\R}(M)$ is the determinant of an integral positive-definite Gram
matrix, and hence is a positive integer.
Thus the following theorem gives a family of sharp cases of Regev and
Stephens-Davidowitz's conjecture.

\begin{theorem}\label{thm:main}
Let $L$ be a positive-definite integral unimodular lattice of rank
$n\leq 32$.
Then, for every $t>0$,
\[
  \Theta_L(t)\leq \Theta_{\Z^n}(t).
\]
Equality for a single $t>0$ occurs if and only if $L\cong\Z^n$; in that
case, equality holds for every $t>0$.
In addition, if $L$ is even unimodular of rank $40$, then
\[
  \Theta_L(t)<\Theta_{\Z^{40}}(t)\qquad(t>0).
\]
\end{theorem}

In prior work~\cite{Kominers}, we proved that no orbit-constant scalar
Poisson-summation certificate can establish \eqref{eq:RSD-conjecture} in
dimensions $n\geq 8$.
The present proof therefore uses a different method: We work with the
finite-dimensional structure of theta series of integral unimodular lattices,
which dates back to Hecke \cite{Hecke1940} and is closely connected to the
shadow theory developed by Elkies \cite{ElkiesLongShadows} and
Rains--Sloane \cite{RainsSloane}.

The proof also must go beyond shell-by-shell comparisons of lattice vector
counts.
Although \eqref{eq:RSD-conjecture} is a pointwise inequality between
Gaussian sums, $\Z^n$ does not in general dominate other integral unimodular
lattices on each norm shell---for example, $E_8$ has $240$ vectors of
norm $2$, whereas $\Z^8$ has only $112$.
Regev and Stephens-Davidowitz~\cite{RSDshell} proved a uniform upper bound
on the size of every norm shell of an integral lattice, and in subsequent
work \cite{KominersShells} we classified the equality cases, showing that
for $n\geq 2$ the shell bound is saturated only by $\Z^n$ at norm $1$ and by
$E_8$ at norm $2$.
Thus Theorem~\ref{thm:main} cannot follow from termwise shell domination;
the inequality must instead come from cancellations among different shells.
The modular-form parametrization makes these cancellations explicit.

Writing
\[
  u(q)=\frac{\Delta_8(q)}{\vt_3(q)^8}
  =\frac{\vt_2(q)^4\vt_4(q)^4}{16\vt_3(q)^8},
  \qquad q=e^{-t},
\]
Jacobi's identity gives $0<u\leq 1/64$.
Hecke's theorem expresses
\[
  \frac{\Theta_L}{\vt_3^n}
\]
as a polynomial in $u$ of degree $\lfloor n/8\rfloor$.
Since $\lfloor n/8\rfloor\leq 5$ in the ranks considered here, this
polynomial has at most five nonconstant coefficients
$a_1,\ldots,a_{\lfloor n/8\rfloor}$.
These coefficients admit concrete lattice-theoretic control: the coefficient
of $u$ is forced once norm-$1$ vectors have been split off; the coefficient
of $u^2$ is bounded using the ADE root system formed by the norm-$2$
vectors; and, in the odd case, the top coefficients in ranks $24$ through
$32$ are constrained by shadow positivity.
The even rank-\(40\) case is handled separately using the two-dimensionality
of the relevant level-one modular-form spaces and the same ADE root-count
bound.
Our proof therefore does not use the classification of unimodular lattices in
these ranks.
Its lattice-theoretic inputs are only the orthogonal splitting of norm-$1$
vectors, the ADE classification of simply laced root systems, and the
nonnegativity of the shadow theta series.

Several other lines of work are closely related to the present result.
The Belfiore--Sol\'e conjecture, originating in coding-theoretic studies of
the Gaussian wiretap channel, predicts that for every unimodular lattice $L$,
the secrecy function
\[
  \Xi_L(y)=\frac{\Theta_{\Z^n}(iy)}{\Theta_L(iy)}
\]
attains its global maximum on the positive imaginary axis at the symmetry
point $y=1$ \cite{BelfioreSole}; verifications for specific lattice families
have been carried out by Ernvall-Hyt\"onen \cite{ErnvallHytonen}, Pinchak
and Sethuraman \cite{PinchakSethuraman}, and Bollauf, Lin, and Ytrehus
\cite{BollaufLinYtrehus}, among others.
Theorem~\ref{thm:main} is related to, but logically distinct from, this
conjecture: we establish the pointwise lower bound $\Xi_L(y)\geq 1$ for
every integral unimodular lattice through rank $32$, but that does not by
itself locate the global maximum of~$\Xi_L$.

The same Gaussian-mass comparison has also been studied more directly.
A recent preprint of Bollauf and Lin \cite{BollaufLin2024} treats the
maximum-theta-series question over unimodular lattices by means of an
analytic criterion that, when verified for a given lattice, implies the
corresponding upper bound.
The present proof is independent of that criterion and gives an unconditional
comparison through rank $32$ via finite-dimensional modular-form methods.

There are also important special cases coming from universal optimality.
The Cohn--Kumar--Miller--Radchenko--Viazovska universal optimality theorem
for $E_8$ and the Leech lattice $\Lambda_{24}$ \cite{CKMRV} implies the two
comparisons $\Theta_{E_8}\leq\Theta_{\Z^8}$ and
$\Theta_{\Lambda_{24}}\leq\Theta_{\Z^{24}}$, via the interpolation-formula
machinery built on Viazovska's modular forms \cite{Viazovska,CKMRV24}.
Thus the cases $L=E_8$ and $L=\Lambda_{24}$ of Theorem~\ref{thm:main} are
not new with the present argument.
What is new here is the uniform statement across all integral unimodular
lattices through rank $32$, including the remaining $23$ Niemeier lattices in
rank $24$ and the odd unimodular lattices in these ranks.

Finally, the work most methodologically similar to our argument is that of
Cohn and Triantafillou \cite{CohnTriantafillou}, who used dual modular forms
in $E_4$, $\Delta$, and the level-one ring to bound sphere-packing densities
in dimensions $12$, $16$, $20$, $28$, and $32$.
The present proof is similar in spirit in that a finite-dimensional
modular-form space is reduced to explicit coefficient inequalities, but the
inequalities here concern Gaussian theta-series comparisons with $\Z^n$, and
the odd unimodular cases require the additional shadow constraints described
above.

The remainder of this paper is organized as follows.
Section~\ref{sec:prelim} fixes theta-function conventions and records the
modular identities used throughout.
Section~\ref{sec:odd} establishes coefficient and shadow constraints for odd
unimodular lattices with no norm-$1$ vectors.
Section~\ref{sec:odd-rank-32} assembles these into the strict inequality
through rank $32$ in the odd norm-$1$-free case.
Section~\ref{sec:even} treats even unimodular lattices through rank $40$.
Section~\ref{sec:main-proof} proves Theorem~\ref{thm:main}.
Section~\ref{sec:scope} describes the first ranks at which the coefficient
method we use requires new input.

\section{Theta Constants and Hecke's Polynomial Form}\label{sec:prelim}

This section collects the analytic input the rest of the paper depends on.
We fix theta-function conventions, introduce the modular function
$u=\Delta_8/\vt_3^8$ that parametrizes the theta-series ratio, and record
three lemmata: the range $u\in(0,1/64]$ (Lemma~\ref{lem:u-range}), the
identities $E_4/\vt_3^8=1-16u$ and $\Delta/\vt_3^{24}=u^2$
(Lemma~\ref{lem:E4Delta}), and the polynomial form of $\Theta_L/\vt_3^n$ in
$u$ (Lemma~\ref{lem:hecke}).

We put
\[
  q=e^{\pi i\tau},\qquad \tau\in\HH,
\]
and later specialize to $\tau=it/\pi$, so that $q=e^{-t}\in(0,1)$.
Thus our $q$ is the square root of the more common modular variable
$Q=e^{2\pi i\tau}$; this convention makes
$\vt_3(q)=\sum_m q^{m^2}$ have $q$-expansion in integer powers of $q$, but
it shifts the $q$-expansions of $E_4$ and $\Delta$ to powers of $q^2$.
The theta constants are
\[
\begin{aligned}
  \vt_2(q)&=\sum_{m\in\Z}q^{(m+1/2)^2},\\
  \vt_3(q)&=\sum_{m\in\Z}q^{m^2},\\
  \vt_4(q)&=\sum_{m\in\Z}(-1)^m q^{m^2},
\end{aligned}
\]
and hence
\[
  \Theta_{\Z^n}(t)=\vt_3(e^{-t})^n.
\]
We furthermore define
\[
  \Delta_8(q)=\frac{\vt_2(q)^4\vt_4(q)^4}{16},
  \qquad
  u(q)=\frac{\Delta_8(q)}{\vt_3(q)^8}.
\]

\begin{lemma}\label{lem:u-range}
For $0<q<1$, we have
\[
  0<u(q)\leq \frac{1}{64}.
\]
\end{lemma}

\begin{proof}
Jacobi's identity \cite[Section~21.22]{WhittakerWatson}
\[
  \vt_3(q)^4=\vt_2(q)^4+\vt_4(q)^4
\]
gives
\[
  u(q)=\frac{1}{16}
  \left(\frac{\vt_2(q)^4}{\vt_3(q)^4}\right)
  \left(\frac{\vt_4(q)^4}{\vt_3(q)^4}\right).
\]
For $0<q<1$ the two factors in parentheses are positive (by the Jacobi
triple product formula, $\vt_4(q)>0$ on $(0,1)$) and have sum $1$; their
product is at most~$1/4$.
\end{proof}

The parameter $u$ is closely tied to the level-one Eisenstein series $E_4$
and the modular discriminant $\Delta$, normalized in the $q=e^{\pi i\tau}$
convention by
\[
  E_4(\tau)=1+240q^2+2160q^4+\cdots,
  \qquad
  \Delta(\tau)=q^2\prod_{m\geq1}(1-q^{2m})^{24}.
\]
The following identities make the relationship explicit and let us
re-express $E_4$ and~$\Delta$ as theta-series multiples of polynomials in
$u$.

\begin{lemma}\label{lem:E4Delta}
For $q=e^{\pi i\tau}$, we have the identities
\[
  \frac{E_4(\tau)}{\vt_3(q)^8}=1-16u(q),
  \qquad
  \frac{\Delta(\tau)}{\vt_3(q)^{24}}=u(q)^2.
\]
\end{lemma}

\begin{proof}
The classical theta identities
\[
  E_4=\frac{1}{2}(\vt_2^8+\vt_3^8+\vt_4^8),
  \qquad
  \Delta=\frac{1}{256}\vt_2^8\vt_3^8\vt_4^8
\]
are standard.
Write $a=\vt_2^4$ and $b=\vt_4^4$.
Since $\vt_3^4=a+b$,
\[
  E_4=\frac{1}{2}\bigl(a^2+(a+b)^2+b^2\bigr)=a^2+ab+b^2=\vt_3^8-ab.
\]
As $ab=16\Delta_8$, this gives $E_4/\vt_3^8=1-16u$.
The discriminant identity gives
\[
  \frac{\Delta}{\vt_3^{24}}
  =\frac{\vt_2^8\vt_4^8}{256\vt_3^{16}}
  =\left(\frac{\vt_2^4\vt_4^4}{16\vt_3^8}\right)^2
  =u^2. \qedhere
\]
\end{proof}

Our coefficient calculations use the following expansions:
\begin{align}
  \vt_3^8&=1+16q+112q^2+448q^3+1136q^4+O(q^5),\notag\\
  \Delta_8&=q-8q^2+28q^3-64q^4+O(q^5),\notag\\
  \frac{1}{\vt_3^8}&=1-16q+144q^2-960q^3+5264q^4+O(q^5),\notag\\
  u&=q-24q^2+300q^3-2624q^4+O(q^5);\label{eq:u-expansion}
\end{align}
these follow from the definitions of $\vt_2$, $\vt_3$, and $\vt_4$, via
direct multiplication.

With the preceding conventions fixed, we turn to the structure theorem for
theta series of integral unimodular lattices.
Let $L$ be a positive-definite integral unimodular lattice of rank $n$, and
set
\[
  \Theta_L(\tau)=\sum_{x\in L} e^{\pi i\|x\|^2\tau}.
\]
We use the same symbol $\Theta_L$ for the Gaussian mass introduced in
Section~\ref{sec:intro} and for this modular theta series; the two
specializations agree under $\tau=it/\pi$, since then
$e^{\pi i\|x\|^2\tau}=e^{-t\|x\|^2}$.
Hecke's theorem, in the form used by Elkies
\cite[p.~644, equations (4)--(7)]{ElkiesLongShadows}, says that $\Theta_L$
is a weighted-homogeneous polynomial in the theta series of $\Z$ and $E_8$
(see also \cite[Chapter~7, Theorem~7]{CS}).

\begin{lemma}\label{lem:hecke}
Let $L$ be a positive-definite integral unimodular lattice of rank $n$.
Then
\[
  \Theta_L=\sum_{j=0}^{\lfloor n/8\rfloor} b_j\,\vt_3^{\,n-8j}E_4^j
\]
for some $b_j\in\C$.
Equivalently, there are real coefficients $a_j$ such that
\begin{equation}\label{eq:u-polynomial}
  \frac{\Theta_L}{\vt_3^n}=1+a_1u+a_2u^2+\cdots+a_\ell u^\ell,
\end{equation}
where $\ell=\lfloor n/8\rfloor$.
\end{lemma}

\begin{proof}
The theta series $\Theta_L$ of a positive-definite integral unimodular
lattice of rank $n$ is a modular form of weight $n/2$ for the theta group
$\Gamma_\theta=\langle S,T^2\rangle$, with the same multiplier system as
$\vt_3^n$.
Hecke's theorem \cite{Hecke1940} says that the space of such forms is spanned
by the monomials
\[
  \vt_3^{\,n-8j}E_4^j,\qquad 0\leq j\leq\lfloor n/8\rfloor,
\]
each of which has the correct weight.
Since $\Theta_L$ lies in this space, it admits the claimed expression.

By Lemma~\ref{lem:E4Delta}, we have $E_4=\vt_3^8(1-16u)$.
Dividing the Hecke expression by $\vt_3^n$ gives a polynomial in $u$ of
degree at most $\ell$---with constant term $1$, since both $\Theta_L$ and
$\vt_3^n$ have constant term $1$ and $u(q)\to0$ as $q\to0$.
The coefficients are real because the Fourier coefficients of the theta
series are real.
\end{proof}

It is convenient to use the equivalent $\Delta_8$-basis.
Setting $a_0=1$ and using $u^j=\Delta_8^j/\vt_3^{8j}$, we may rewrite
\eqref{eq:u-polynomial} as
\begin{equation}\label{eq:delta8-basis}
  \Theta_L=\sum_{j=0}^{\ell}a_j\,\vt_3^{\,n-8j}\Delta_8^j,
\end{equation}
with the same coefficients $a_j$.

\section{Odd Unimodular Lattices}\label{sec:odd}

This section establishes the coefficient and shadow constraints for odd
unimodular lattices with no norm-$1$ vectors.
Section~\ref{sec:odd-rank-32} assembles the constraints into the strict
inequality through rank $32$; then in Section~\ref{sec:main-proof}, we
extend that result to apply across all integral unimodular lattices, by
splitting off norm-$1$ vectors.

\subsection{The first two coefficients}

The Hecke polynomial form (Lemma~\ref{lem:hecke}) expresses
$\Theta_L/\vt_3^n$ as a polynomial in $u$.
In the rank ranges we care about, the coefficients of $u$ and $u^2$ are
determined by direct lattice data (rather than by shadow positivity, which we
use for higher coefficients).
The next lemma reads them off in terms of $n$ and the number of norm-$2$
vectors.

\begin{lemma}\label{lem:coefficients}
Let $L$ be a positive-definite integral unimodular lattice of rank $n$ with
no vectors of norm $1$.
Let
\[
  N_2=\#\{x\in L:\|x\|^2=2\}.
\]
In the expansion \eqref{eq:u-polynomial}, if $n\geq 8$, then
\[
  a_1=-2n.
\]
If $n\geq 16$, then
\[
  a_2=N_2+2n^2-46n.
\]
\end{lemma}

\begin{proof}
The absence of norm-$1$ vectors gives
\[
  \Theta_L=1+N_2q^2+O(q^3).
\]
Since $\vt_3=1+2q+O(q^4)$,
\[
  \vt_3^n=1+2nq+2n(n-1)q^2+O(q^3),
\]
and inverting the power series gives
\[
  \frac{1}{\vt_3^n}=1-2nq+(2n^2+2n)q^2+O(q^3).
\]
Multiplying $\Theta_L$ by this reciprocal, we have
\begin{align*}
  \frac{\Theta_L}{\vt_3^n}
  &=(1+N_2q^2+O(q^3))
    (1-2nq+(2n^2+2n)q^2+O(q^3))\\
  &=1-2nq+(N_2+2n^2+2n)q^2+O(q^3).
\end{align*}
On the other hand, Lemma~\ref{lem:hecke} expresses $\Theta_L/\vt_3^n$ as a
polynomial in $u$,
\[
  \frac{\Theta_L}{\vt_3^n}=1+a_1u+a_2u^2+\cdots+a_\ell u^\ell.
\]
The $q$-expansion \eqref{eq:u-expansion} gives $u=q-24q^2+O(q^3)$ and
$u^2=q^2+O(q^3)$, while $u^j=O(q^j)=O(q^3)$ for $j\geq3$, so
\[
  1+a_1u+a_2u^2+O(q^3)
  =1+a_1q+(a_2-24a_1)q^2+O(q^3).
\]
The two $q$-expansions of $\Theta_L/\vt_3^n$ must agree.
Coefficient comparison gives $a_1=-2n$ and
\[
  a_2+48n=N_2+2n^2+2n,
\]
which is the claimed formula for $a_2$.
\end{proof}

\subsection{An ADE root-count bound}

The coefficient $a_2$ depends on $N_2$ linearly, so an upper bound on $N_2$
gives an upper bound on $a_2$.
Norm-$2$ vectors form a simply laced root system, and from rank $16$ onward
the irreducible ADE components are uniformly controlled by a single
rank-dependent quantity.

\begin{lemma}\label{lem:root-bound}
Let $R$ be a simply laced root system of total rank at most $n$.
If $n\geq16$, then
\[
  |R|\leq 2n(n-1).
\]
Consequently, if $L$ is a positive-definite integral lattice of rank
$n\geq16$, then the number $N_2$ of norm-$2$ vectors of $L$ satisfies
$N_2\leq 2n(n-1)$.
\end{lemma}

\begin{proof}
If $\alpha,\beta$ are norm-$2$ vectors in $L$, then $(\alpha,\beta)\in\Z$ by
integrality, and the Cauchy-Schwarz inequality gives
$|(\alpha,\beta)|\leq 2$, with equality only when $\beta=\pm\alpha$.
Moreover the reflection
\[
  \beta\longmapsto \beta-(\alpha,\beta)\alpha
\]
preserves $L$ (since $(\alpha,\beta)\in\Z$) and preserves the norm.
Hence the norm-$2$ vectors form a reduced simply laced root system, which by
the ADE classification (see, e.g., \cite[Section~1.4]{Ebeling};
\cite[Chapter~4]{CS}) is a direct sum of irreducible components of types
$A_r$ ($r\geq 1$), $D_r$ ($r\geq 4$), and the exceptional types $E_6$,
$E_7$, $E_8$.
Their root counts are
\[
  |A_r|=r(r+1),\qquad |D_r|=2r(r-1),
\]
and
\[
  |E_6|=72,
  \qquad |E_7|=126,
  \qquad |E_8|=240.
\]
For each irreducible component of rank \(r\), the number of roots is at most
$2r(n-1)$ when $n\geq16$:
Indeed, for \(A_r\) this says \(r(r+1)\leq2r(n-1)\), which follows from
\(r\leq n\) and \(n\geq3\).
For \(D_r\) this says \(2r(r-1)\leq2r(n-1)\), which follows from
\(r\leq n\).
For \(E_6,E_7,E_8\), the inequalities are
\[
72\leq 2\cdot6(n-1),\qquad
126\leq2\cdot7(n-1),\qquad
240\leq2\cdot8(n-1),
\]
which are all valid for \(n\geq16\).
Summing over components of total rank at most \(n\) gives
\[
|R|\leq2(n-1)\sum_i r_i\leq2n(n-1). \qedhere
\]
\end{proof}

Combining Lemma~\ref{lem:coefficients} with Lemma~\ref{lem:root-bound}, we
see that a rank $n\geq16$ lattice with no norm-$1$ vectors must satisfy
\begin{equation}\label{eq:a2-bound}
  a_2\leq 2n(n-1)+2n^2-46n=4n(n-12).
\end{equation}

\subsection{The shadow constraints}

The bound~\eqref{eq:a2-bound} on $a_2$ alone is enough to isolate
$\Theta_L/\vt_3^n$ through rank $23$.
Beyond that, the polynomial $\Theta_L/\vt_3^n$ has higher-degree coefficients
$a_3$, $a_4$ that the root system does not control directly.
These coefficients are pinned down instead by shadow positivity, which
expresses an additional theta series of $L$ as a Fourier series with
nonnegative coefficients.

Let $L$ be an odd unimodular lattice.
The even sublattice
\[
  L_0=\{x\in L:\|x\|^2\in 2\Z\}
\]
has index $2$ in $L$.
The \emph{shadow} of $L$ is the coset union
\[
  S(L)=L_0^*\setminus L,
\]
where $^*$ denotes the dual lattice.
The theta series
\[
  \Theta_{S(L)}=\sum_{x\in S(L)}q^{\|x\|^2}
\]
has nonnegative coefficients, possibly in fractional powers of $q$.

\begin{lemma}\label{lem:shadow-formula}
Let $L$ be a positive-definite odd unimodular lattice of rank $n$, written in
the $\Delta_8$-basis as in \eqref{eq:delta8-basis}.
Then
\begin{equation}\label{eq:shadow-formula}
  \Theta_{S(L)}
  =
  \sum_{j=0}^{\ell}
  \frac{(-1)^j}{16^j}a_j\,
  \vt_2(q)^{\,n-8j}\vt_4(q^2)^{8j},
\end{equation}
where, as in Lemma~\ref{lem:hecke}, we take $\ell=\lfloor n/8\rfloor$.
\end{lemma}

\begin{proof}
Elkies \cite[p.~644, equation (3)]{ElkiesLongShadows} records the modular
transformation
\[
  \Theta_L\!\left(-\frac{1}{\tau}+1\right)
  =\left(\frac{\tau}{i}\right)^{n/2}\Theta_{S(L)}(\tau)
\]
relating $\Theta_L$ to $\Theta_{S(L)}$.
By Lemma~\ref{lem:hecke}, $\Theta_L$ is a polynomial in $\vt_3$ and $E_4$.
Elkies's transformation acts on this polynomial expression by sending
$\vt_3\mapsto\vt_2$ and fixing $E_4$; substituting yields
$\Theta_{S(L)}$ as the corresponding polynomial in $\vt_2$ and $E_4$.
The first substitution comes from $S(\Z)=\Z+\tfrac{1}{2}$, which gives
$\Theta_{S(\Z)}=\vt_2$; the second from the fact that $E_4=\Theta_{E_8}$ is
a modular form for the full group $\SL_2(\Z)$ (since $E_8$ is even
unimodular), so the shadow operation, viewed as a weight-normalized slash
action, fixes it.

We compute the image of $\Delta_8$ under this substitution.
By Lemma~\ref{lem:E4Delta},
\[
  \Delta_8=\frac{\vt_3^8-E_4}{16},
\]
which is sent to $(\vt_2^8-E_4)/16$.
The duplication identity
\begin{equation}\label{eq:duplication}
  \vt_4(q^2)^2=\vt_3(q)\vt_4(q)
\end{equation}
(see \cite[Section~20.7]{DLMF}) and the Jacobi identity
$\vt_3^4=\vt_2^4+\vt_4^4$ together give
\[
  \vt_4(q^2)^8=\vt_3(q)^4\vt_4(q)^4=\vt_2^4\vt_4^4+\vt_4^8.
\]
By Lemma~\ref{lem:E4Delta} and the Jacobi identity,
\[
  E_4=\vt_3^8-16\Delta_8
  =(\vt_2^4+\vt_4^4)^2-\vt_2^4\vt_4^4
  =\vt_2^8+\vt_2^4\vt_4^4+\vt_4^8,
\]
hence
\[
  \vt_4(q^2)^8=E_4-\vt_2(q)^8.
\]
Thus $\Delta_8$ is sent to $-\vt_4(q^2)^8/16$, and substituting into the
$\Delta_8$-basis expression \eqref{eq:delta8-basis} for $\Theta_L$ gives
\eqref{eq:shadow-formula}.
\end{proof}

As a consistency check on the normalization in \eqref{eq:shadow-formula}:
when $L=\Z^n$, we have $a_0=1$ and $a_j=0$ for $j\geq1$, so the formula
reduces to $\Theta_{S(\Z^n)}=\vt_2(q)^n$.
This is correct, since $S(\Z^n)=(\Z+\tfrac{1}{2})^n$ has theta series
$\vt_2^n$ by the definition of $\vt_2$.

The leading term of the $j$-th summand in \eqref{eq:shadow-formula} has
exponent $(n-8j)/4$ and leading coefficient
\begin{equation}\label{eq:shadow-leading-coefficient}
  \frac{(-1)^j}{16^j}a_j2^{n-8j},
\end{equation}
since $\vt_2(q)=2q^{1/4}+O(q^{9/4})$ and $\vt_4(q^2)=1+O(q^2)$.

For $24\leq n\leq31$, we have $\ell=\lfloor n/8\rfloor=3$.
Recall that $S(L)=L_0^*\setminus L$, so $0\notin S(L)$, and the constant
term of $\Theta_{S(L)}$ is zero.
By \eqref{eq:shadow-leading-coefficient}, the leading coefficient of the
$j=3$ summand is
\[
  -\frac{a_3}{16^3}2^{\,n-24}.
\]
If $n=24$, then the $j=3$ term is constant, and its coefficient must vanish;
hence $a_3=0$.
If $25\leq n\leq31$, the $j=3$ term is the unique lowest-degree term in the
shadow theta series (with exponent $(n-24)/4>0$), so nonnegativity of its
coefficient gives $a_3\leq0$.
Thus
\begin{equation}\label{eq:a3-nonpositive}
  a_3\leq0\qquad(24\leq n\leq31).
\end{equation}
For $n=32$, the $j=4$ term is constant with coefficient $a_4/16^4$; again,
since $0\notin S(L)$, we have
\begin{equation}\label{eq:a4-zero}
  a_4=0.
\end{equation}
After this vanishing, the $j=4$ summand contributes nothing to
\eqref{eq:shadow-formula}, so the next--lowest-degree term is the $j=3$ term
with leading coefficient $-a_3\cdot 2^{8}/16^3=-a_3/16$, and nonnegativity
gives
\begin{equation}\label{eq:a3-nonpositive-32}
  a_3\leq0\qquad(n=32).
\end{equation}

\section{The Odd Case Through Rank $32$}\label{sec:odd-rank-32}

We now assemble the lemmata of Section~\ref{sec:odd} into a strict
Gaussian-mass inequality.
The strategy is dimensional: at rank $n\leq 32$, the polynomial
$\Theta_L/\vt_3^n$ has degree at most $4$, so the proof reduces to bounding
the four nonconstant coefficients $a_1,a_2,a_3,a_4$.
The constraints from Section~\ref{sec:odd} give $a_1=-2n$, force
$a_3\leq 0$ and (at $n=32$) $a_4=0$ when these enter, and bound $a_2$ above
sharply enough that, although $a_2u^2$ may be positive, the linear term
$-2nu$ dominates it on $u\in(0,1/64]$.

\begin{proposition}\label{prop:odd-no-one}
Let $L$ be a positive-definite odd unimodular lattice of positive rank
$n\leq32$ with no norm-$1$ vectors.
Then
\[
  \Theta_L(t)<\Theta_{\Z^n}(t)\qquad(t>0).
\]
\end{proposition}

\begin{proof}
Let $q=e^{-t}$ and $u=u(q)$.
Lemma~\ref{lem:u-range} gives $0<u\leq1/64$.
Since $\Theta_{\Z^n}(t)=\vt_3(q)^n$, the strict inequality
$\Theta_L(t)<\Theta_{\Z^n}(t)$ is equivalent to
$\Theta_L/\vt_3^n - 1 < 0$, and it is this difference that we bound by cases
on $n$ below.
(Both $\Theta_L$ and $\vt_3^n$ are positive, so the ratio itself lies in
$(0,1)$ whenever the inequality holds; the negative quantities appearing in
the sequel are the differences $\Theta_L/\vt_3^n - 1$, not the ratios
themselves.)

For $1\leq n<8$, suppose $L$ is as in the proposition's hypotheses.
Then \eqref{eq:u-polynomial} has degree $\lfloor n/8\rfloor=0$, so
$\Theta_L=\vt_3^n$.
The $q$-coefficient of $\Theta_L=\sum_{x\in L}q^{\|x\|^2}$ counts the
norm-$1$ vectors in $L$, which is $0$ by hypothesis; but the $q$-coefficient
of $\vt_3^n$ is $2n>0$.
This contradiction shows that no such $L$ exists in positive rank below $8$,
and the proposition holds automatically in that range.

Now, assume $8\leq n\leq15$.
Then $\ell=1$, and Lemma~\ref{lem:coefficients} gives
\[
  \frac{\Theta_L}{\vt_3^n}=1-2nu<1.
\]

Next, assume $16\leq n\leq23$.
Then $\ell=2$, and \eqref{eq:a2-bound} gives
\[
  \frac{\Theta_L}{\vt_3^n}-1
  =-2nu+a_2u^2
  \leq
  u\left(-2n+\frac{4n(n-12)}{64}\right)
  =\frac{nu(n-44)}{16}<0.
\]

And next, assume $24\leq n\leq31$.
Then $\ell=3$.
By \eqref{eq:a3-nonpositive}, the cubic term $a_3u^3$ is nonpositive, and
the estimate from the previous case (which holds for every $n\geq16$ with
$n<44$) gives
\[
  \frac{\Theta_L}{\vt_3^n}-1
  =-2nu+a_2u^2+a_3u^3
  \leq -2nu+a_2u^2
  \leq \frac{nu(n-44)}{16}<0.
\]

Finally, when $n=32$, equations \eqref{eq:a4-zero} and
\eqref{eq:a3-nonpositive-32} give
\[
  \frac{\Theta_L}{\vt_3^{32}}-1
  =-64u+a_2u^2+a_3u^3
  \leq -64u+a_2u^2.
\]
By \eqref{eq:a2-bound}, $a_2\leq4\cdot32\cdot20=2560$, and therefore
\[
  -64u+a_2u^2
  \leq u\left(-64+\frac{2560}{64}\right)
  =-24u<0. \qedhere
\]
\end{proof}

\begin{remark}\label{rem:low-rank-examples}
Odd unimodular lattices without any norm-$1$ vectors appear already in
moderate rank.
Elkies \cite[pp.~645--646]{ElkiesLongShadows} lists those of rank at most
$23$ that minimize the number of norm-$2$ vectors, including the rank-$12$
lattice with root system $D_{12}$.
The rank-$23$ shorter Leech lattice $O_{23}$ is distinguished instead by
having no norm-$1$ or norm-$2$ vectors.
\end{remark}

\section{Even Unimodular Lattices Through Rank $40$}\label{sec:even}

For even unimodular lattices, the theta series is a level-one modular form.
Through weight $20$---equivalently, rank $40$---the space of such forms is at
most two-dimensional, and the ADE root-count bound from Section~\ref{sec:odd}
suffices to control the relevant Fourier coefficients.
The argument in the even case is therefore more straightforward than in the
odd case; in particular, shadow theory need not enter.

Let $L$ be even unimodular of rank $n=8m$.
Then $\Theta_L$ is a level-one modular form of weight $4m$.
For $1\leq m\leq5$, the space $M_{4m}(\SL_2(\Z))$ of such forms has
dimension~$1$ for $m=1,2$ and dimension~$2$ for $m=3,4,5$; the general
dimension formula is recalled in Section~\ref{sec:scope}.
Put
\[
  e=\frac{E_4}{\vt_3^8}=1-16u;
\]
by Lemma~\ref{lem:u-range}, we have
\begin{equation}\label{eq:e-range}
  \frac{3}{4}\leq e<1.
\end{equation}

\begin{proposition}\label{prop:even}
Let $L$ be a positive-definite even unimodular lattice of rank $8m$ with
$1\leq m\leq5$.
Then
\[
  \Theta_L(t)<\Theta_{\Z^{8m}}(t)\qquad(t>0).
\]
\end{proposition}

\begin{proof}
As in the proof of Proposition~\ref{prop:odd-no-one}, the strict inequality
$\Theta_L(t)<\Theta_{\Z^{8m}}(t)$ is equivalent to
$\Theta_L/\vt_3^{8m}<1$, since $\Theta_{\Z^{8m}}=\vt_3^{8m}$; we establish
this ratio bound below.

For $m=1,2$, the space $M_{4m}(\SL_2(\Z))$ is one-dimensional.
Hence $\Theta_L=E_4^m$, and
\[
  \frac{\Theta_L}{\vt_3^{8m}}=e^m<1.
\]

For $m=3,4,5$, the space $M_{4m}(\SL_2(\Z))$ is two-dimensional, with basis
$E_4^m$ and $E_4^{m-3}\Delta$.
Thus
\[
  \Theta_L=E_4^m+cE_4^{m-3}\Delta.
\]
Let $r$ be the number of norm-$2$ vectors in $L$.
The coefficient of $q^2$ in $\Theta_L$ is $r$, while $E_4^m$ contributes
$240m$ and $E_4^{m-3}\Delta$ contributes $1$.
Hence
\[
  c=r-240m.
\]
The norm-$2$ vectors form a simply laced root system of rank at most $8m$.
Lemma~\ref{lem:root-bound} gives
\[
  r\leq2(8m)(8m-1),
\]
and therefore
\begin{equation}\label{eq:c-bound}
  c\leq2(8m)(8m-1)-240m=128m(m-2).
\end{equation}
Using Lemma~\ref{lem:E4Delta},
\[
  \frac{\Theta_L}{\vt_3^{8m}}
  =e^m+ce^{m-3}u^2
  \leq e^m+128m(m-2)e^{m-3}u^2.
\]
Since $u=(1-e)/16$, the upper bound is, respectively,
\[
  e^3+\frac{3}{2}(1-e)^2,
  \qquad
  e^4+4e(1-e)^2,
  \qquad
  e^5+\frac{15}{2}e^2(1-e)^2
\]
for $m=3,4,5$.

For $e\in[3/4,1)$, the desired strict inequality
$\Theta_L/\vt_3^{8m}<1$ reduces to showing that the upper bound above is
strictly less than $1$, equivalently, that $1$ minus the upper bound is
strictly positive.
We factor:
\begin{align*}
  1-e^3-\frac{3}{2}(1-e)^2
    &=\frac{1-e}{2}(2e^2+5e-1),\\
  1-e^4-4e(1-e)^2
    &=(1-e)(e^3+5e^2-3e+1),\\
  1-e^5-\frac{15}{2}e^2(1-e)^2
    &=\frac{1-e}{2}(2e^4+17e^3-13e^2+2e+2).
\end{align*}
The first factor on the right ($\frac{1-e}{2}$ or $1-e$) is strictly
positive on $[3/4,1)$.
It therefore suffices to show that each second factor is strictly positive
on $[3/4,1]$.
We do this by verifying that each (i) is positive at $e=3/4$ and (ii) has
positive derivative on $[3/4,1]$, and thus is monotone increasing on the
interval.

\smallskip
\noindent\emph{Positivity at $e=3/4$.}
Direct substitution gives the values
\[
  \tfrac{31}{8},
  \qquad
  \tfrac{127}{64},
  \qquad
  \tfrac{511}{128}
\]
for the three second factors, all positive.

\smallskip
\noindent\emph{Monotonicity on $[3/4,1]$.}
The derivatives of the second factors are
\[
  4e+5,
  \qquad
  3e^2+10e-3,
  \qquad
  8e^3+51e^2-26e+2,
\]
respectively.
The first two are evidently positive on $[3/4,1]$: $4e+5\geq 8$, and
$3e^2+10e-3$ has positive derivative there and takes the value $99/16$ at
$e=3/4$.
For the third, write $g(e)=8e^3+51e^2-26e+2$; its derivative
$g'(e)=24e^2+102e-26$ takes the value $64$ at $e=3/4$ and is strictly
increasing on $[3/4,1]$, so $g$ is monotone increasing there.
Since $g(3/4)=233/16>0$, we have $g(e)>0$ throughout $[3/4,1]$.

\smallskip
Combining the preceding calculations, we see that each second factor is
strictly positive on $[3/4,1]$.
Hence, the upper bound on $\Theta_L/\vt_3^{8m}$ is strictly less than $1$
throughout $[3/4,1)$, and the strict inequality
$\Theta_L(t)<\Theta_{\Z^{8m}}(t)$ follows for $m=3,4,5$.
\end{proof}

\section{Proof of the Main Theorem}\label{sec:main-proof}

The two strict inequalities established so far---Proposition~\ref{prop:odd-no-one}
for odd unimodular lattices of rank $n\leq 32$ without norm-$1$ vectors, and
Proposition~\ref{prop:even} for even unimodular lattices of rank
$\leq 40$---almost give Theorem~\ref{thm:main} directly.
The remaining step is to dispose of norm-$1$ vectors in the odd case, which
we do by splitting them off as orthogonal copies of $\Z$.

\begin{lemma}\label{lem:norm-one-splitting}
Let $L$ be a positive-definite integral unimodular lattice.
Then there is an orthogonal decomposition
\begin{equation}\label{eq:norm-one-splitting}
  L\cong \Z^r\oplus M,
\end{equation}
where $r\geq0$ and $M$ is a positive-definite integral unimodular lattice
with no norm-$1$ vectors.
\end{lemma}

\begin{proof}
If $v\in L$ has norm $1$, then for every $x\in L$ the integrality of $L$
gives $(x,v)\in\Z$, and
\[
  x-(x,v)v\in L\cap v^\perp.
\]
Thus
\[
  L=\Z v\oplus (L\cap v^\perp)
\]
orthogonally.
The Gram matrix is block diagonal with a $1$ in the $\Z v$ block; since $L$
is unimodular, $L\cap v^\perp$ is again an integral unimodular lattice.
Iteration gives the decomposition \eqref{eq:norm-one-splitting}.
This splitting is also recorded by Elkies \cite[p.~643]{ElkiesLongShadows}.
\end{proof}

\begin{proof}[Proof of Theorem~\ref{thm:main}]
Let $L$ have rank $n\leq32$, and apply
Lemma~\ref{lem:norm-one-splitting} to obtain $L\cong\Z^r\oplus M$ with $M$
as above.

If $M=0$, then $L\cong\Z^n$ and equality
$\Theta_L(t)=\Theta_{\Z^n}(t)$ holds for every $t>0$.

If $M\neq0$, then $M$ is either odd or even.
In the odd case, Proposition~\ref{prop:odd-no-one} gives
\[
  \Theta_M(t)<\Theta_{\Z^{\rank(M)}}(t)
  \qquad(t>0).
\]
In the even case, the rank of $M$ is a positive multiple of $8$ (since every
positive-definite even unimodular lattice has rank divisible by $8$) and at
most $32$, so Proposition~\ref{prop:even} gives the same strict inequality.
Theta series multiply under orthogonal direct sums, and hence
\[
  \Theta_L(t)
  =\Theta_{\Z^r}(t)\Theta_M(t)
  <\Theta_{\Z^r}(t)\Theta_{\Z^{n-r}}(t)
  =\Theta_{\Z^n}(t)
\]
for every $t>0$.

In particular, if equality $\Theta_L(t)=\Theta_{\Z^n}(t)$ holds for any
single $t>0$, then $M=0$ and so $L\cong\Z^n$; and conversely
$L\cong\Z^n$ gives equality for every $t>0$.
This proves the rank $n\leq 32$ assertion of Theorem~\ref{thm:main} along
with its equality clause.

The rank-$40$ even assertion is Proposition~\ref{prop:even} with $m=5$,
applied directly to $L$ itself, since an even lattice contains no vector of
norm $1$.
\end{proof}

\section{Boundary of the Coefficient Method}\label{sec:scope}

The argument succeeds through rank $32$ in the odd case and rank $40$ in the
even case because the relevant spaces of modular forms are small enough that
the lattice-theoretic data we use---norm-$1$ absence, the ADE root system,
and the shadow---controls the first few Fourier coefficients of $\Theta_L$.
At the next rank in each case, a new Fourier coefficient enters that is not
pinned down by these data, and the method would require new input.
This section identifies the obstruction in each case, then discusses what the
rank-$32$ boundary looks like geometrically.

The first odd rank not covered by the result of
Proposition~\ref{prop:odd-no-one} is $33$.
Then $\ell=4$, and the shadow's leading term comes from $j=4$.
Formula \eqref{eq:shadow-leading-coefficient} gives the lower bound
$a_4\geq0$, whereas the theta-ratio contains the term $a_4u^4$ with positive
sign:
\[
  \frac{\Theta_L}{\vt_3^{33}}
  =1-66u+a_2u^2+a_3u^3+a_4u^4.
\]
The rank-$32$ proof used the special fact that the top shadow term was
constant and therefore vanished.
No corresponding upper bound for $a_4$ follows at rank $33$ from root count
and shadow positivity alone.

The first even rank not covered by Proposition~\ref{prop:even} is $48$.
For weights $4m$, the dimension formula is
\[
  \dim \bigl(M_{4m}(\SL_2(\Z))\bigr)=\left\lfloor\frac{m}{3}\right\rfloor+1,
\]
because $4m\not\equiv2\pmod{12}$.
Thus $M_{24}(\SL_2(\Z))$ has dimension $3$:
\[
  M_{24}=\langle E_4^6,\ E_4^3\Delta,\ \Delta^2\rangle.
\]
The root number determines the coefficient of $E_4^3\Delta$, but not the
coefficient of $\Delta^2$; controlling that coefficient, or replacing it
with a different positivity input, is the next finite-dimensional obstruction
for this method.

Beyond the modular-form obstruction, the rank-$32$ boundary marks a
substantive shift in lattice geometry, on two distinct fronts.
First, the population of unimodular lattices in this range becomes enormous:
King's mass formula \cite{King} computes the mass of even unimodular
lattices of rank $32$ with prescribed root system, including the rootless
case, and the resulting count is very large.
Second, within the even-unimodular critical-point problem studied by
Heimendahl et al., rank $32$ is the first rank at which local maxima of
Gaussian energy appear: Heimendahl, Marafioti, Thiemeyer, Vallentin, and
Zimmermann \cite{Heimendahl} studied the critical-point structure of the
Gaussian energy
\[
  E_c(\Lambda)=\sum_{x\in\Lambda\setminus\{0\}}e^{-c\|x\|^2}
\]
on the moduli space of unit-covolume lattices, proving that every even
unimodular lattice through rank $24$ is a critical point of $E_c$ with
explicitly computed Morse index, and exhibiting rank-$32$ even unimodular
lattices that are local maxima of $E_c$ on the moduli space.
Theorem~\ref{thm:main} establishes that no such local maximum exceeds
$\Z^{32}$ globally.
Pushing the present coefficient method to handle the residual cusp
coefficient at rank $48$, and locating the rank-$32$ lattices that realize
specific local maxima of the moduli-space energy, are natural next targets.

\providecommand{\bysame}{\leavevmode\hbox to3em{\hrulefill}\thinspace}
\providecommand{\MR}{\relax\ifhmode\unskip\space\fi MR }
\providecommand{\MRhref}[2]{%
  \href{http://www.ams.org/mathscinet-getitem?mr=#1}{#2}
}
\providecommand{\href}[2]{#2}

\end{document}